\documentclass[11pt,a4paper,reqno]{amsart}

\usepackage{lmodern}
\usepackage{amsmath,amssymb,mathtools}
\usepackage{microtype}
\usepackage{needspace}
\usepackage[margin=28mm]{geometry}
\usepackage[hidelinks]{hyperref}

\makeatletter
\renewcommand{\@setauthors}{%
  \par\vspace{1.2ex}%
  \begingroup
  \centering
  \def\authorcolumn##1##2{%
    \begin{minipage}[t]{0.47\textwidth}
      \centering
      {\small\scshape ##1\par}%
      \vspace{0.5ex}%
      \footnotesize
      School of Mathematical Sciences\par
      Xiamen University\par
      Xiamen, China\par
      \href{mailto:##2}{\texttt{##2}}\par
    \end{minipage}%
  }%
  \noindent
  \authorcolumn{Jungang Chen}{jgchen@stu.xmu.edu.cn}%
  \hfill
  \authorcolumn{Jiaxin Xie}{jiaxinxie@stu.xmu.edu.cn}\par
  \endgroup
  \vspace{1ex}%
}
\let\@setaddresses\relax
\makeatother

\numberwithin{equation}{section}
\newtheorem{theorem}{Theorem}[section]
\newtheorem*{theorem*}{Theorem}
\newtheorem*{conjecture}{Conjecture}
\newtheorem{lemma}[theorem]{Lemma}

\theoremstyle{definition}

\theoremstyle{remark}

\newcommand{\sig}{\sigma}

\hypersetup{
  pdftitle={A direct inductive proof of the sharp Merino--Welsh threshold for matroids},
  pdfsubject={A direct inductive proof of the sharp Merino--Welsh threshold for matroids},
  pdfkeywords={Tutte polynomial, matroid, Merino-Welsh inequality, optimal constant},
  pdfauthor={Jungang Chen and Jiaxin Xie}
}

\title[A direct inductive proof of the sharp Merino--Welsh threshold for matroids]{A direct inductive proof of the sharp Merino--Welsh threshold for matroids}
\author{Jungang Chen}
\address{School of Mathematical Sciences\\
  Xiamen University\\
  Xiamen, China}
\email{jgchen@stu.xmu.edu.cn}
\author{Jiaxin Xie}
\address{School of Mathematical Sciences\\
  Xiamen University\\
  Xiamen, China}
\email{jiaxinxie@stu.xmu.edu.cn}
\thanks{Corresponding author: Jungang Chen.}
\date{}
\subjclass[2020]{Primary 05B35; Secondary 05C31}

\keywords{Tutte polynomial, matroid, Merino--Welsh inequality, optimal constant, parallel extension}

\begin{document}

\begin{abstract}
Motivated by the Merino--Welsh conjecture, we consider the smallest $c\ge0$, denoted by $c_*$, for which the inequality $T_M(c,0)T_M(0,c)\ge T_M(1,1)^2$ holds for every loopless and coloopless finite matroid $M$. The counterexamples constructed by Beke, Cs\'aji, Csikv\'ari, and Pituk [\emph{Adv. Math.} \textbf{446} (2024), 109674] give the lower bound $x_0$, where $x_0\approx2.22668$ is the largest real root of the polynomial $x^3-9(x-1)$. Later, Csikv\'ari [\emph{European J. Combin.} \textbf{137} (2026), 104402] improved the known upper bound for this constant to $2.35$ and then conjectured that the above inequality holds at $c=x_0$. This conjecture was recently proved by Liu (2026). We give an alternative direct inductive proof that $c_*=x_0$, without computer-assisted finite verification.
\end{abstract}

\maketitle

\section{Introduction}

Let $M$ be a finite matroid with ground set $E=E(M)$ and rank function $r_M$. Its Tutte polynomial is defined by
\[
T_M(x,y)=\sum_{A\subseteq E}(x-1)^{r_M(E)-r_M(A)}(y-1)^{|A|-r_M(A)}.
\]
For $c\ge0$, define $D_c(M)=T_M(c,0)T_M(0,c)-T_M(1,1)^2$, where $T_M(1,1)$ is the number of bases of $M$; we denote the family of bases of the matroid by $\mathcal{B}(M)$ and write $b(M):=|\mathcal{B}(M)|$. The multiplicative Merino--Welsh conjecture for matroids (see \cite[Conjecture~1.2(3)]{KnauerEtAl2018}) states that every loopless and coloopless matroid $M$ satisfies $D_2(M)\ge0$. This conjecture generalizes the multiplicative conjecture proposed by Conde and Merino \cite{CondeMerino2009} for graphs.
In this paper, we define
\[
c_* = \inf\{c\ge0:D_c(M)\ge0\text{ holds for every loopless and coloopless matroid }M\}.
\]
Since the coefficients of the Tutte polynomial are nonnegative, $D_c(M)$ is a continuous, nondecreasing function of $c$ on $[0,\infty)$; together with the known upper bounds for $c_*$ given below, this shows that the above set of parameters is precisely the closed ray $[c_*,\infty)$.

All matroids in this paper are finite, and we allow the empty matroid, whose Tutte polynomial is $1$. Unless otherwise stated, our matroid terminology and notation follow the second edition of Oxley's \emph{Matroid Theory} \cite{Oxley2011}.

Using deletion--contraction induction together with parallel and series reductions, Jackson \cite{Jackson2010} established the inequality for $c=3$. The matroid version of this result is $c_*\le3$; see \cite{Csikvari2026}. Beke, Cs\'aji, Csikv\'ari, and Pituk \cite{BCsCPPermutation2024} introduced the permutation Tutte polynomial, reducing inequalities for the Tutte polynomial of a matroid to estimates on the fundamental graph of each basis (that is, the local basis-exchange bipartite graph). Combining Harris's correlation inequality with a transfer lemma, they obtained $c_*\le2.9243$. On the other hand, let $q(x)=x^3-9(x-1)$ and denote its largest real root by $x_0$. Beke, Cs\'aji, Csikv\'ari, and Pituk \cite{BCsCP2024} replaced each element of a uniform matroid by two parallel elements and compared the exponential growth rates of the relevant Tutte polynomial evaluations. In this way, for every $0\le c<x_0$, they constructed infinitely many loopless and coloopless matroids satisfying $D_c(M)<0$, yielding the lower bound $c_*\ge x_0$; in particular, the multiplicative Merino--Welsh conjecture for matroids is false. Subsequently, Csikv\'ari \cite{Csikvari2026} refined the local degree estimates in the permutation Tutte polynomial framework and combined them with finite numerical checks to obtain $c_*\le2.355$. After excluding one further local structure, Csikv\'ari \cite{Csikvari2026} noted that the same method gives $c_*\le2.35$ and finally proposed the following conjecture:

\begin{conjecture}
For every loopless and coloopless matroid $M$, we have $D_{x_0}(M)\ge 0.$
\end{conjecture}

Liu \cite{Liu2026} recently proved this conjecture using a weighted broken-circuit comparison, the Cunningham--Edmonds decomposition, and a finite table of rational replacements verified by exact computation. We give an alternative proof by deletion--contraction induction, using parallel-class reductions and an elementary scalar inequality. Our proof requires neither a finite replacement table nor computer-assisted verification.

\section{Some lemmas}\label{sec:preliminaries}

We assume that the reader is familiar with the basic concepts and properties of matroids. The relevant properties used in this paper can be found in the second edition of Oxley's \emph{Matroid Theory} \cite[Sections~1.1--1.4, 2.1, 3.1, 4.1--4.2, and 5.4]{Oxley2011}; for the basic properties of the Tutte polynomial, see \cite{BrylawskiOxley1992}.

For a loopless matroid $M$, fix a linear order on $E(M)$ and let $\mathcal{N}(M)$ denote the family of subsets containing no broken circuit, where a broken circuit is obtained from a circuit by deleting its largest element. The broken-circuit theorem \cite{BrylawskiOxley1992} gives
\begin{equation}\label{eq:nbc-expansion}
T_M(c,0)=\sum_{I\in\mathcal{N}(M)}(c-1)^{r(M)-|I|}.
\end{equation}

The deletion--contraction form of the following estimate appears in \cite[Lemma~2.1]{Liu2026}.

\begin{lemma}\label{lem:axis}
If $M$ is loopless, $e\in E(M)$, and $c>1$, then
\begin{equation}\label{eq:axis-contraction}
T_M(c,0)\ge c\,T_{M/e}(c,0).
\end{equation}
Dually, if $M$ is coloopless, then
\begin{equation}\label{eq:axis-deletion}
T_M(0,c)\ge c\,T_{M\backslash e}(0,c).
\end{equation}
\end{lemma}

\begin{proof}
If $M/e$ has a loop, then $T_{M/e}(c,0)=0$, and \eqref{eq:axis-contraction} follows immediately from the nonnegativity of the coefficients. We may therefore assume that $M/e$ is loopless. Choose a linear order on $E(M)$ in which $e$ is the largest element, and use the induced order on $E(M)-\{e\}$. We claim that
\begin{equation}\label{eq:nbc-inclusion}
I\in\mathcal{N}(M/e)\quad\Longrightarrow\quad
I,\ I\cup\{e\}\in\mathcal{N}(M).
\end{equation}
It suffices to verify that $I\cup\{e\}\in\mathcal{N}(M)$. Suppose that it contains the broken circuit $C-\{\max C\}$ corresponding to a circuit $C$ of $M$. If $e\in C$, then $e=\max C$ and $C-\{e\}\subseteq I$. However, $C-\{e\}$ is dependent in $M/e$, whereas the no-broken-circuit set $I$ is independent in this matroid, a contradiction; in the remaining case $e\notin C$, we have $C-\{\max C\}\subseteq I$. Since a dependent set remains dependent after contraction, $C$ is dependent in $M/e$. Choose a circuit $C'$ of $M/e$ contained in $C$; then
\[
C'-\{\max C'\}\subseteq C-\{\max C\}\subseteq I.
\]
Thus, $I$ contains a broken circuit of $M/e$, again giving a contradiction. Therefore, \eqref{eq:nbc-inclusion} holds. The two families of sets given by \eqref{eq:nbc-inclusion} are disjoint, and $r(M/e)=r(M)-1$. Since $c-1>0$, equation~\eqref{eq:nbc-expansion} yields \eqref{eq:axis-contraction}:
\begin{align*}
T_M(c,0)
&\ge\sum_{I\in\mathcal{N}(M/e)}
 \left((c-1)^{r(M)-|I|}+(c-1)^{r(M)-1-|I|}\right)\\
&=c\sum_{I\in\mathcal{N}(M/e)}(c-1)^{r(M/e)-|I|}=cT_{M/e}(c,0).
\end{align*}
Applying \eqref{eq:axis-contraction} to $M^*$ and using $(M\backslash e)^*=M^*/e$ gives \eqref{eq:axis-deletion}.
\end{proof}

\begin{lemma}\label{lem:parallel-reduction}
Let $M$ be a connected, loopless, and coloopless matroid, and let $P\subsetneq E(M)$ be a parallel class with $|P|=k\ge2$. Choose $e\in P$ and define $N=M\backslash(P-\{e\})$, $L=N\backslash e=M\backslash P$, and $Q=N/e$. Then $N$ and $Q$ are both loopless and coloopless, while $L$ is loopless. Moreover,
\begin{equation}\label{eq:parallel-polynomial}
T_M(x,y)=T_L(x,y)+(1+y+\cdots+y^{k-1})T_Q(x,y),
\end{equation}
and
\begin{equation}\label{eq:parallel-bases}
b(M)=b(L)+kb(Q).
\end{equation}
\end{lemma}

\begin{proof}
Since $M$ is loopless, both $N$ and $L$ are loopless. By the maximality of $P$, no other element of $N$ is parallel to $e$, so $Q=N/e$ is loopless. Keeping one element of a parallel class and deleting the others preserves connectivity, so $N$ is connected. Since $P\subsetneq E(M)$, the matroid $N$ has at least two elements; if $N$ had a coloop, that coloop would form a connected component on its own, which is impossible, so $N$ is coloopless. Since contraction does not create new coloops, $Q$ is also coloopless.

To compute $T_M$, let $N_j$ be the matroid obtained by replacing $e$ in $N$ by $j$ parallel elements. For $j\ge2$, deleting one of these elements gives $N_{j-1}$, while contracting it gives the direct sum of $j-1$ loops and $Q$. Thus,
\[
T_{N_j}(x,y)=T_{N_{j-1}}(x,y)+y^{j-1}T_Q(x,y).
\]
Iterating from $T_{N_1}=T_N=T_L+T_Q$ gives \eqref{eq:parallel-polynomial}; evaluating at $(1,1)$ then yields \eqref{eq:parallel-bases}.
\end{proof}

\begin{lemma}\label{lem:two-coloops}
Let $M$ be loopless and coloopless, and suppose that every element of $M$ belongs to a two-element circuit or a two-element cocircuit. Let $P$ be a parallel class with $|P|\ge2$. If $L=M\backslash P$ has a coloop, then it has at least two coloops.
\end{lemma}

\begin{proof}
Choose a coloop $f$ of $L$. No other element of $M$ is parallel to $f$: such an element in $P$ would contradict the maximality of $P$, whereas one outside $P$ would give a two-element circuit of $L$ through $f$, contradicting the fact that $f$ is a coloop; hence, by hypothesis, $f$ must be in series with some element $g$, that is, $\{f,g\}$ is a cocircuit of $M$. If $g\in P$, choose $h\in P-\{g\}$; the circuit $\{g,h\}$ and the cocircuit $\{f,g\}$ would then intersect only in $g$, contrary to orthogonality, so $g\notin P$. Thus, $f,g$ are parallel in the dual matroid $M^*$. Since $f$ is a loop of $L^*=M^*/P$, the element $g$ is also a loop of $M^*/P$, and hence a coloop of $L$. The two coloops $f,g$ are distinct, proving the claim.
\end{proof}

\begin{lemma}\label{lem:basis-injection}
Let $M$ be coloopless, and suppose that it has a nonloop element $e$. Set $L=M\backslash e$ and $Q=M/e$. Writing $\mathcal{J}$ for the set of coloops of $L$, we have
\begin{equation}\label{eq:basis-injection}
b(Q)\ge |\mathcal{J}|\,b(L).
\end{equation}
\end{lemma}

\begin{proof}
If $\mathcal{J}=\varnothing$, the conclusion is immediate. We may therefore assume that $\mathcal{J}\ne\varnothing$.
Since $e$ is not a coloop of $M$, we have
$r(L)=r(M)$, so every basis of $L$ is also a basis of $M$.
Moreover, since $\mathcal{J}$ is the set of coloops of $L$, we have
$\mathcal{J}\subseteq B$ for every $B\in\mathcal{B}(L)$. Consider the assignment
\[
\begin{aligned}
\Phi\colon \mathcal{B}(L)\times\mathcal{J}
&\longrightarrow \mathcal{B}(Q),\\
(B,j)&\mapsto B\setminus\{j\}.
\end{aligned}
\]
To prove \eqref{eq:basis-injection}, it suffices to show that $\Phi$ is well-defined and injective.

We first verify that the map is well-defined. Take any $B\in\mathcal{B}(L)$ and $j\in\mathcal{J}$; since $j$ is not a coloop of $M$, there exists $B'\in\mathcal{B}(M)$ such that $j\notin B'$. Since $B$ is also a basis of $M$, applying the basis-exchange axiom to $B,B'$ gives an element $j'\in B'\setminus B$ such that $(B\setminus\{j\})\cup\{j'\}\in\mathcal{B}(M)$. If $j'\ne e$, then $(B\setminus\{j\})\cup\{j'\}$ is also a basis of $L$ that omits $j$, contradicting the fact that $j$ is a coloop of $L$; therefore, $j'=e$, and hence $(B\setminus\{j\})\cup\{e\}\in\mathcal{B}(M)$. Since $e$ is a nonloop, the characterization of bases under contraction gives $B\setminus\{j\}\in\mathcal{B}(M/e)=\mathcal{B}(Q)$. Thus, $\Phi$ is well-defined.

We next verify injectivity. Suppose that $\Phi(B_1,j_1)=\Phi(B_2,j_2)$, that is, $B_1\setminus j_1=B_2\setminus j_2$; since $j_1,j_2$ are both coloops of $L$, we have $j_1\in B_2,j_2\in B_1$, so $j_1=j_2$ and consequently $B_1=B_2$. Hence, $\Phi$ is injective.
\end{proof}

\begin{lemma}\label{lem:scalar}
Let $k\ge2$ be an integer, and let $a,b,u,v$ be positive real numbers satisfying $a\ge(x_0-1)b$ and $v\ge(x_0-1)u$. Then
\[
(a+b)\bigl(u+\sig_k v\bigr)
\ge\left(\sqrt{au}+k\sqrt{bv}\right)^2,
\]
where $\sig_k=1+x_0+x_0^2+\cdots+x_0^{k-1}.$
\end{lemma}

\begin{proof}
We proceed by induction on $k\ge2$. When $k=2$, we have $\sig_2=1+x_0$, and the hypotheses give
\[
\sqrt{bu}\le\frac{\sqrt{av}}{x_0-1},\qquad
(a+b)v\le\frac{x_0}{x_0-1}av.
\]
Thus,
\[
\begin{aligned}
&(a+b)(u+\sig_2 v)-\left(\sqrt{au}+2\sqrt{bv}\right)^2\\
={}&\left(2\sqrt{av}-\sqrt{bu}\right)^2-(3-x_0)(a+b)v\\
\ge{}&\left(\frac{(2x_0-3)^2}{(x_0-1)^2}-\frac{x_0(3-x_0)}{x_0-1}\right)av\\
={}&\frac{x_0^3-9(x_0-1)}{(x_0-1)^2}av=0.
\end{aligned}
\]
Now suppose that the conclusion holds for some integer $k\ge2$. Using $\sig_{k+1}=\sig_k+x_0^k$ and the induction hypothesis, we obtain
\[
\begin{aligned}
&(a+b)(u+\sig_{k+1}v)
  -\left(\sqrt{au}+(k+1)\sqrt{bv}\right)^2\\
={}&(a+b)(u+\sig_k v)-\left(\sqrt{au}+k\sqrt{bv}\right)^2\\
&\qquad{}+x_0^k(a+b)v-(2k+1)bv-2\sqrt{abuv}\\
\ge{}&x_0^k(a+b)v-(2k+1)bv-2\sqrt{abuv}\\
>{}&2k(a+b)v-(2k+1)bv-2\sqrt{abuv}\\
={}&2kav-au-2bv+\left(\sqrt{au}-\sqrt{bv}\right)^2\\
\ge{}&2kav-au-2bv>2kav-av-2av>0.
\end{aligned}
\]
This completes the induction.
\end{proof}

\section{Main theorem}\label{sec:main-proof}
We give an alternative proof of the following theorem of Liu \cite[Theorem~1.1]{Liu2026}.

\begin{theorem*}
  $c_*=x_0$.
\end{theorem*}
\begin{proof}
In view of the known lower bound $c_*\ge x_0$, it suffices to prove that every loopless and coloopless matroid $M$ satisfies $D_{x_0}(M)\ge0$. We proceed by induction on $|E(M)|$. The Tutte polynomial of the empty matroid is $1$, so the conclusion holds. Now let $M$ be nonempty, and suppose that the conclusion holds for all loopless and coloopless matroids with fewer elements.

If $M$ is disconnected, then each of its connected components is still loopless and coloopless and has fewer elements. Applying the induction hypothesis to each component and using the factorization of the Tutte polynomial over direct sums shows that $M$ satisfies the required inequality. We may therefore assume that $M$ is connected.

If some element $f$ of $M$ belongs to neither a two-element circuit nor a two-element cocircuit, then both $M\backslash f$ and $M/f$ are loopless and coloopless. Applying the induction hypothesis to these two minors, each of which has fewer elements, and then using the deletion--contraction relation for the Tutte polynomial and the Cauchy--Schwarz inequality gives the conclusion. Thus, it remains to consider the case in which every element of $M$ belongs to a two-element circuit or a two-element cocircuit; in this case, $M$ has a nontrivial parallel class or series class, and the required inequality is invariant under duality, so we may assume without loss of generality that $M$ has a parallel class $P$ with $|P|=k\ge2$.

If $P=E(M)$, then $M$ must be the uniform matroid of rank one, whose Tutte polynomial is
\[
T_M(x,y)=x+y+\cdots+y^{k-1}.
\]
Since $x_0\ge2$ and $2^j\ge2j$, we obtain
\[
T_M(x_0,0)T_M(0,x_0)
=x_0\sum_{j=1}^{k-1}x_0^j\ge4\sum_{j=1}^{k-1}j
=2k(k-1)\ge k^2=b(M)^2,
\]
so it remains to consider the case $P\subsetneq E(M)$. Define $N,L,Q$ as in Lemma~\ref{lem:parallel-reduction}; then $N$ and $Q$ are loopless and coloopless, and $L$ is loopless. Set $a=T_L(x_0,0)$ and $b=T_Q(x_0,0)$; further, set $u=T_L(0,x_0)$ and $v=T_Q(0,x_0)$, and write $s=b(L)$ and $t=b(Q)$.
Equations~\eqref{eq:parallel-polynomial} and~\eqref{eq:parallel-bases} give
\begin{equation}\label{eq:main-evaluations}
T_M(x_0,0)=a+b,\qquad T_M(0,x_0)=u+\sig_k v,\qquad b(M)=s+kt.
\end{equation}
Since $e$ is neither a loop nor a coloop of $N$, the deletion--contraction relation and Lemma~\ref{lem:axis} give
\begin{equation}\label{eq:main-ratios}
a\ge(x_0-1)b,\qquad v\ge(x_0-1)u.
\end{equation}

If $L$ is coloopless, then both $L$ and $Q$ are loopless and coloopless matroids with fewer elements, so the induction hypothesis gives
\[
au\ge s^2,\qquad bv\ge t^2.
\]
Thus, combining \eqref{eq:main-evaluations} with Lemma~\ref{lem:scalar} directly yields
\[
T_M(x_0,0)T_M(0,x_0)=(a+b)\bigl(u+\sig_k v\bigr)\ge\left(\sqrt{au}+k\sqrt{bv}\right)^2\ge(s+kt)^2=b(M)^2.
\]
This proves the conclusion in this case.

Now suppose that $L$ has a coloop, so that $u=0$. By Lemma~\ref{lem:two-coloops}, the set $\mathcal{J}$ of coloops of $L$ has at least two elements; applying Lemma~\ref{lem:basis-injection} to $N$ then gives
\begin{equation}\label{eq:t-ge-2s}
t\ge |\mathcal{J}|s\ge2s.
\end{equation}
The matroid $Q$ has fewer elements and is loopless and coloopless, so the induction hypothesis gives $bv\ge t^2$. Applying Lemma~\ref{lem:scalar} to the quadruple $(x_0-1,1,1,x_0-1)$ gives
\[
x_0\bigl(1+(x_0-1)\sig_k\bigr)\ge\left(\sqrt{x_0-1}+k\sqrt{x_0-1}\right)^2=(k+1)^2(x_0-1).
\]
That is, $x_0^{k+1}\ge (k+1)^2(x_0-1)$, and using this bound gives
\begin{equation}\label{eq:coloop-scalar}
x_0\sig_k=\frac{x_0^{k+1}-x_0}{x_0-1}\ge(k+1)^2-\frac{x_0}{x_0-1}>(k+1)^2-2>\left(k+\frac12\right)^2,
\end{equation}
where the last step uses $(k+1)^2-2-(k+\frac12)^2=k-\frac54>0$. Combining \eqref{eq:main-evaluations}, \eqref{eq:main-ratios}, and~\eqref{eq:coloop-scalar}, we obtain
\[
T_M(x_0,0)T_M(0,x_0)=(a+b)\sig_k v\ge x_0\sig_k bv\ge x_0\sig_k t^2>\left(k+\frac12\right)^2t^2\ge(s+kt)^2=b(M)^2,
\]
where the last step uses \eqref{eq:t-ge-2s}. The conclusion also holds in this case, completing the induction.
\end{proof}

\section*{Data availability}
Data sharing is not applicable to this article, as no datasets were generated or analysed in this theoretical study.

\section*{Funding}
This research received no specific grant from any funding agency in the public, commercial, or not-for-profit sectors.

\section*{Use of artificial intelligence}
The authors used ChatGPT (Astra, OpenAI; \url{https://chatgpt.com}) in early September 2026 for language editing and translation, literature searches, and exploratory discussions of parts of the proofs. They also used Codex (OpenAI, model \texttt{gpt-6-astra}; \url{https://openai.com/codex/}) in October 2026 for manuscript editing, literature comparisons, and feedback on the proofs. The authors checked each step of the arguments they adopted and take full responsibility for the mathematical content and the final manuscript.

\end{document}